\documentclass{article} 
\usepackage{graphicx}           
\usepackage{amsthm}
\usepackage{amssymb}
\usepackage{amsfonts}
\usepackage{amsmath}
\usepackage{makeidx}
\usepackage{mathtools}

\usepackage{enumerate}

\usepackage{tikz} 
\usepackage{comment}
\usetikzlibrary{arrows}

\newcommand{\scong}{\cong_s}
\newcommand{\id}{\text{id}}

\newcommand{\ZZ}{\mathbb{Z}}
\newcommand{\NN}{\mathbb{N}}
\newcommand{\MM}{\mathbb{M}}
\newcommand{\zot}{\mathcal{Z}}
\newcommand\restr[2]{{
  \left.\kern-\nulldelimiterspace 
  #1 
  \right|_{#2} 
  }}

\newcommand{\shift}{\nu}
\newcommand{\shiftt}{\mu}

\newtheorem{proposition}{Proposition}
\newtheorem{theorem}{Theorem}
\newtheorem{corollary}{Corollary}

\theoremstyle{definition}

\theoremstyle{remark}
\newtheorem{remark}{Remark}

\theoremstyle{question}
\newtheorem{question}{Question}

\title{On the Conjugacy Problem of Cellular Automata\thanks{Research supported by the Academy of Finland Grant 296018.}}

\author{Jarkko Kari\\
	University of Turku\\
	20500 Turku, Finland
	\and
	Joonatan Jalonen\thanks{Author's research supported by the Finnish Cultural Foundation.}\thanks{email: jsjalo@utu.fi}\\
	University of Turku\\
	20500 Turku, Finland
}

\begin{document}

\maketitle

\begin{abstract}
Cellular automata are topological dynamical systems. We consider the problem of deciding whether two cellular automata are conjugate or not. We also consider deciding strong conjugacy, that is, conjugacy by a map that commutes with the shift maps. We show that the following two sets of pairs of one-dimensional one-sided cellular automata are recursively inseparable:
\begin{enumerate}[(i)]
\item pairs where the first cellular automaton has strictly higher entropy than the second one, and
\item pairs that are strongly conjugate and both have zero topological entropies.
\end{enumerate}
This implies that the following decision problems are undecidable: Given two one-dimensional one-sided cellular automata $F$ and $G$: Are $F$ and $G$ conjugate? Is $F$ a factor of $G$? Is $F$ a subsystem of $G$? All of these are undecidable in both strong and weak variants (whether the homomorphism is required to commute with the shift or not, respectively).

We also prove the same results for reversible two-dimensional  cellular automata.

\end{abstract}

\section{Introduction}
Cellular automata were designed to be a model for natural computing. A cellular automaton consists of simple devices, namely finite state automata, on an integer lattice. The finite state automata update their states synchronously depending only on the states of the automata in some finite neighborhood. The computational complexity of cellular automata arises from parallelism. In \cite{Kari94} it was proved that injectivity (which for cellular automata is equivalent to bijectivity) and surjectivity of two-dimensional cellular automata is undecidable, while the same problems are decidable for one-dimensional cellular automata \cite{Amoroso72}. As a more recent development, in \cite{DennunzioFormentiWeiss14} a multidimensional version of the closing property of cellular automata was defined and proven undecidable; also this property is known to be decidable for one-dimensional cellular automata \cite{Kurka03}.

The Curtis-Lyndon-Hedlund Theorem \cite{Hedlund69} says that cellular automata can equivalently be defined as the shift commuting endomorphisms of a full shift. This leads to the fruitful study of cellular automata as topological dynamical systems. Some dynamical properties are known to be undecidable for cellular automata, for example equicontinuity (can be seen by combining results from \cite{Kari92} and \cite{Kurka97} as shown in \cite{DurandFormentiVarouchas03}, and remains undecidable even among reversible cellular automata \cite{KariOllinger08}), left expansivity \cite{KariLukkarila09}, and mixingness \cite{Lukkarila10}.

In this paper we focus on the conjugacy problem of cellular automata, i.e. the problem of deciding whether two cellular automata are conjugate or not. In \cite{Epperlein17} it was proved that conjugacy of one-dimensional periodic cellular automata is decidable. They also conjectured that conjugacy for general one-dimensional cellular automata is undecidable. This article is an extended version of the conference paper \cite{JalonenKari18}, in which this conjecture was proven. This proof is presented in Section \ref{One-dimensional case}, where we prove that given a pair of one-dimensional one-sided cellular automata, it is recursively inseparable whether the first cellular automaton has strictly larger topological entropy than the second one, or whether the two are strongly conjugate and have zero topological entropies. This implies that (strong) conjugacy, being a (strong) factor, and being a (strong) subsystem are undecidable properties for one-dimensional one- and two-sided cellular automata. Here ``strong'' means that the corresponding homomorphism is shift commuting. The essential tool of our proof is the undecidability of nilpotency of one-dimensional one-sided cellular automata \cite{Kari92}.

One of the interesting questions left open is whether (strong) conjugacy remains undecidable when restricted to reversible one-dimensional cellular automata. Reversible cellular automata are those that are bijective and whose inverse is also a cellular automaton; it turns out that for cellular automata reversibility and bijectivity are equivalent \cite{Hedlund69}. In Section \ref{Two-dimensional case} we prove the same inseparability result for reversible two-dimensional two-sided cellular automata as in Section \ref{One-dimensional case} for one-dimensional one-sided cellular automata. This again implies that (strong) conjugacy, being a (strong) factor, and being a (strong) subsystem, are undecidable properties. The undecidability of conjugacy of two-dimensional cellular automata was already proved in \cite{Epperlein17}. In our proof the important ingredients come from \cite{Kari94} and \cite{Meyerovitch08}. In \cite{Kari94}  a special kind of tile set was constructed to show that reversibility of two-dimensional cellular automata is undecidable. Using this tile set it was proved in \cite{Meyerovitch08} that there are two-dimensional cellular automata that have strictly positive finite entorpies. Our construction also gives a positive answer to \cite[Question 6.1.]{Meyerovitch08} which asks whether there exist reversible two-dimensional cellular automata with strictly positive finite entropies.

\section{Preliminaries}

\label{Preliminaries}
\subsection{Symbolic dynamics}

Zero is considered a natural number, i.e. $0\in\mathbb{N}$. For two integers $i,j\in\mathbb{Z}$ such that $i < j$ the integer interval from $i$ to $j$ is denoted $[i,j]=\{i,i+1,\dots,j\}$, we also denote $[i,j)=\{i,i+1,\dots j-1\}$ and $(i,j]=\{i+1,\dots,j\}$. 
Composition of functions $f:X\rightarrow Y$ and $g:Y\rightarrow Z$ is written as $gf$, and defined by $(gf)(x)=g(f(x))$ for all $x\in X$. The set of all functions $X\rightarrow Y$ is denoted by $Y^X$. We use the notation $\MM$ when it does not matter whether $\ZZ$ or $\NN$ is used.

Let $A$ be a finite set called the \emph{alphabet} or the \emph{state set}. An element $c\in A^{\MM^d}$ is a \emph{configuration}. Configurations are maps that assign letters, or states, to cells of an integer lattice. We denote $c(u)=c_u$ for $u\in\MM^d$. For any subset $D\subseteq \MM^d$ we denote by $c_D$ the restriction of $c$ to the domain $D$.  Let $D\subset\MM^d$ be finite and $p\in A^D$. The set $[p]=\{c\in A^\mathbb{M^d}\mid c_D=p\}$ is called a \emph{cylinder set}.
Let $\lVert u\rVert=\max\{\lvert u_1\rvert,\dots,\lvert u_d\rvert\}$ for all $u\in\MM^d$. We consider $A^{\MM^d}$ to be a metric space with the metric
\[
\delta(c,e) =
\begin{cases}
 2^{-\min\left(\{\lVert u\rVert\mid c_u\neq e_u\}\right)},&\text{ if }c\neq e \\
0,&\text{ if } c=e
\end{cases}
\]
for all $c,e\in A^{\MM^d}$. This defines a compact space. Cylinders form a countable clopen (open and closed) base of the topology that the metric $\delta$ induces.

For every $u\in\MM^d$ we define a \emph{shift map} $\sigma_u:A^{\MM^d}\rightarrow A^{\MM^d}$ by $\sigma_u(c)_v=c_{u+v}$ for all $c\in A^{\MM^d}$ and $v\in\MM^d$. The maps $\sigma_u$ are continuous. A subset $X\subseteq A^{\MM^d}$ is called a \emph{subshift} if it is closed and invariant under $\sigma_u$ for all $u\in\MM^d$. The subshift $A^{\MM^d}$ is called a \emph{$d$-dimensional full shift}. For any $n\in\NN$ we denote $\mathcal{C}_n=[0,n)^d$. A configuration $c\in A^{\MM^d}$ \emph{avoids} $p\in A^{\mathcal{C}_n}$ if $\sigma_u(c)_{\mathcal{C}_n}\neq p$ for all $u\in\MM^d$; otherwise $p$ \emph{appears} in $c$. Let $S\subseteq \bigcup_{n\in\NN} A^{\mathcal{C}_n}$, and let $X_S$ be the set of configurations that avoid $S$, i.e. $X_S=\{c\in A^{\MM^d}\mid\forall p\in S:c\text{ avoids }p\}$. It is well-known that the topological definition of subshifts is equivalent to saying that there exists a set $S$ of forbidden patterns such that $X=X_S$. If there exists a finite set $S$ such that $X=X_S$, then $X$ is a \emph{subshift of finite type}. If there exists a subshift of finite type $Y$ and a continuous shift commuting map $F$ such that $F(Y)=X$, then $X$ is a \emph{sofic shift}.

\subsection{Cellular automata}

A \emph{cellular automaton} is a dynamical system $(X,F)$ where $X\subseteq A^{\MM^d}$ is a subshift and $F:X\rightarrow X$ is a continuous map which commutes with the shift maps, i.e. $F\sigma_u=\sigma_uF$ for all $u\in\MM^d$. We will only consider cellular automata over full shifts. When $\mathbb{M}=\mathbb{N}$, the cellular automaton is called \emph{one-sided} and when $\mathbb{M}=\mathbb{Z}$, the cellular automaton is called \emph{two-sided}. There has not been much study of  one-sided cellular automata for $d>1$. We will often refer to a cellular automaton by the function name alone, i.e. talk about the cellular automaton $F$. Let $D\subset\MM^d$ be a finite set and let $G_l:A^D\rightarrow A$. Define $G:A^\mathbb{M}\rightarrow A^\mathbb{M}$ by $G(c)_u = G_l((\sigma_u(c))_D)$ for all $u\in\MM^d$. Now $G$ is continuous and commutes with the shift maps, so it is a cellular automaton. The set $D$ is a \emph{local neighborhood} of $G$ and the function $G_l$ is a \emph{local rule} of $G$. According to the Curtis-Hedlund-Lyndon Theorem \cite{Hedlund69} every cellular automaton is defined by a local rule. Let $r\in\mathbb{N}$ be a number such that $D\subseteq[-r,r]^d$, then $r$ is a \emph{radius} of $G$.

A cellular automaton $F$ is \emph{reversible} if there exists a cellular automaton $G$ such that $FG=\id=GF$ where $\id:A^{\ZZ^d}\rightarrow A^{\ZZ^d}$ is defined by $\id(c)=c$ for all $c\in A^{\ZZ^2}$ (which is clearly a cellular automaton). Classical results \cite{Hedlund69} say that every bijective cellular automaton is  reversible, and further that injective cellular automaton is also surjective so that injectivity, bijectivity, and reversibility, are equivalent conditions for cellular automata.

Let $(A^{\MM^d},F)$ and $(B^{\MM^d},G)$ be two cellular automata. Let $H:A^{\MM^d}\rightarrow B^{\MM^d}$ be a \emph{homomorphism} from $(A^{\MM^d},F)$ to $(B^{\MM^d},G)$, i.e. a continuous map such that $HF=GH$. If $H$ is surjective, it is a \emph{factor map} and $(B^{\MM^d},G)$ is a \emph{factor} of  $(A^{\MM^d},F)$. If $H$ is injective, it is an \emph{embedding} and $(B^{\MM^d},G)$ is a \emph{subystem} of  $(A^{\MM^d},F)$. If $H$ is bijective, it is a \emph{conjugacy} and $(A^{\MM^d},F)$ and  $(B^{\MM^d},G)$ are \emph{conjugate}, which we denote by $(A^{\MM^d},F)\cong (B^{\MM^d},G)$. If $H:A^{\MM^d}\rightarrow B^{\MM^d}$ also commutes with the shift maps, then it is a \emph{strong homomorphism}. We define \emph{strong factor}, \emph{strong subsystem}, and \emph{strongly conjugate}, when the corresponding homomorphism is a strong homomorphism. If $F$ and $G$ are strongly conjugate, we denote $(A^{\MM^d},F)\scong (A^{\MM^d},G)$.

Orbits of a cellular automaton $F:A^{\MM^d}\rightarrow A^{\MM^d}$ are often considered as \emph{space-time diagrams}. These are $(d+1)$-dimensional configurations defined as follows
\[
st(F)=\{ (s^{(i)})_{i\in\NN}\in (A^{\MM^d})^{\NN}\mid \forall i\in\NN: s^{(i+1)}=F(s^{(i)})\}.
\]
If $F$ is surjective, then we can also consider $st(F)\subseteq (A^{\MM^d})^\ZZ$. The space-time diagrams are especially convenient for one-dimensional cellular automata, since it is often easier to get a sense of the dynamics from the static two-dimensional picture rather than the real-time simulations of a cellular automaton. In figures we will have time advancing downwards. Let $\mathcal{R}_F(n,t)$ denote the $\mathcal{C}_n\times[0,t)$ patterns that appear in $st(F)$, i.e. appear in some configurations of $st(F)$. The \emph{(topological) entropy} of $F$ is
\[
h(A^{\MM^d},F)=\lim_{n\rightarrow\infty}\lim_{t\rightarrow\infty}\frac{1}{t}\log_2\lvert\mathcal{R}_F(n,t)\rvert.
\]
We need the following basic property of entropy:
\begin{proposition}{(\cite[Proposition 2.88.]{Kurka03})}
\label{entropy-of-subsystem-and-factor}
If $(B^{\MM^d},G)$ is a subsystem or a factor of $(A^{\MM^d},F)$, then $h(B^{\MM^d},G)\leq h(A^{\MM^d},F)$. It follows that if $(A^{\MM^d},F)$ and $(B^{\MM^d},G)$ are conjugate, then $h(A^{\MM^d},F)=h(B^{\MM^d},G)$.
\end{proposition}

The \emph{direct product} of cellular automata $(A^{\MM^d},F)$ and $(B^{\MM^d},G)$ is $(A^{\MM^d}\times B^{\MM^d}, F\times G)$, where $F\times G:A^{\MM^d}\times B^{\MM^d}\rightarrow A^{\MM^d}\times B^{\MM^d}$ is defined by $(F\times G)(c,e)=(F(c),G(e))$. It is immediate from the definition of entropy that $h(A^{\MM^d}\times B^{\MM^d},F\times G)=h(A^{\MM^d},F)+h(B^{\MM^d},G)$, since $\lvert\mathcal{R}_{F\times G}(n,t)\rvert=\lvert\mathcal{R}_F(n,t)\rvert\cdot\lvert\mathcal{R}_G(n,t)\rvert$.

\section{One-dimensional one-sided cellular automata}

Our first main result considers one-dimensional cellular automata. We give our proof for one-sided cellular automata since the two-sided variant easily follows from this. To this end, in this section we present some simple and well-known basic properties of one-dimensional one-sided cellular automata. After the general results, we discuss a simple example of a reversible one-sided cellular automaton which was also considered in \cite{DartnellMaassSchwartz03}. We show that this cellular automaton has positive entropy and that when restricted to finite words in a certain way, it defines a cyclic permutation of these finite words.

\subsection{Preliminary results about one-dimensional one-sided cellular automata}

Let $F:A^\NN\rightarrow A^\NN$ be a cellular automaton. For every $n\in\mathbb{N}$ we define the \emph{$n$\textsuperscript{th} trace subshift of $F$} as
\[
\tau_n(F)=\big\{e\in \big(A^{n}\big)^\mathbb{N}\mid \exists c\in A^\NN:\forall i\in\mathbb{N}: e_i=\big(F^i(c)\big)_{[0,n)}\big\}.
\]
For surjective cellular automata we can consider $\tau_n(F)$ as a subshift of $(A^n)^\ZZ$. Let $X\subseteq A^\MM$ be a subshift, then we denote $p_n(X)=\lvert\{u\in A^n\mid\exists c\in X: c_{[0,n)}=u\}\rvert$, i.e. the number of different words of length $n$ that appear in $X$. Then the entropy of $F$ can be expressed as
\[
h(A^{\NN},F) = \lim_{n\rightarrow\infty} \lim_{t\rightarrow\infty}\frac{1}{t}\log_2 p_t(\tau_n(F)).
\]
Let $r$ be a radius of $F$. Since $p_n(\tau_{r+1}(F))=\lvert A\rvert\cdot p_n(\tau_r(F))$ for one-sided cellular automata, we get the following:

\begin{proposition}
\label{upper-bound-on-entropy}
Let $F:A^\mathbb{N}\rightarrow A^\mathbb{N}$ be a cellular automaton with radius $r$. Then the entropy of $(A^\NN,F)$ is given by $h(A^\NN,F)=\lim_{t\rightarrow\infty}\frac{1}{t}\log_2 p_t(\tau_r(F))$.
\end{proposition}

For a state $a\in A$ we denote $a^\omega\in A^\NN$ the configuration such that $a^\omega (i)=a$ for all $i\in \NN$. A state $q\in A$ is \emph{quiescent} if $F(q^\omega)=q^\omega$. A cellular automaton is \emph{nilpotent} if there exists a quiescent state $q$ such that for every $c\in A^\NN$ there exists $n\in\mathbb{N}$ such that $F^n(c)=q^\omega$. It is known that for cellular automata nilpotency implies uniform nilpotency:

\begin{proposition}{(\cite{CulikPachlYu89})}
\label{uniformlyNilpotent}
Let $F:A^\NN\rightarrow A^\NN$ be a nilpotent cellular automaton. Then there exists $n\in\mathbb{N}$ such that for all $c\in A^\NN$ it holds that $F^n(c)=q^\omega$.
\end{proposition}

Let $(A^\NN,F)$ be a cellular automaton whose local neighborhood contains cells $0$ and $1$. Then a state $s\in A$ is \emph{spreading} if the local rule maps every neighborhood containing $s$ to $s$. Such a state spreads in the sense that if $c_i=s$ for some $c\in A^\NN$ and $i\in\NN\setminus\{0\}$ then $F(c)_i=s$ and $F(c)_{i-1}=s$. Clearly a spreading state is quiescent. We need the following result, which follows from a simple compactness argument.

\begin{proposition}
\label{spacetime-diagram-without-spreading-state}
Let $F:A^\NN\rightarrow A^\NN$ be a cellular automaton that is not nilpotent, and let $s\in A$ be a spreading state. Then there exists $c\in A^\NN$ such that $F^i(c)_j\neq s$ for all $i,j\in\NN$.
\end{proposition}
\begin{proof}
For every $n\in\NN$ there exists $c^{(n)}\in A^\NN$ such that $F^i(c^{(n)})_j\neq s$ for every $(i,j)\in\{(x,y)\in\NN^2\mid x,y\leq n\}$, since otherwise the appearing states $s$ would spread and $F$ would be nilpotent. By compactness the sequence $(c^{(n)})_{n\in\NN}$ has a converging subsequence $(c^{(i)})_{i\in\mathcal{I}}$, and the limit of this sequence, say $c$, has that $F^i(c)_j\neq s$ for all $i,j\in\NN$ as was claimed.
\end{proof}

Consider a one-sided reversible cellular automaton $F:A^{\mathbb{N}}\rightarrow A^{\mathbb{N}}$ such that both $F$ and its inverse $F^{-1}$ have radius $1$. In many cases this restriction for radius is not a serious one as every reversible cellular automaton is conjugate (though maybe not \emph{strongly} conjugate) to such a cellular automaton through some grouping of cells. Notice that a cellular automaton $F:A^\NN\rightarrow A^\NN$ naturally defines a cellular automaton $F:A^\ZZ\rightarrow A^\ZZ$, however $F$ might be reversible over $A^\ZZ$ but non-reversible over $A^\NN$ (shift map being the canonical counter example). It is easy to see that for every fixed $a\in A$ the map $F(\_a):A\rightarrow A,x\mapsto F(xa)$ has to be a permutation if $F$ is reversible over $A^\NN$; we will denote this permutation with $\rho_a$. Not every set of permutations $\{\rho_a\}_{a\in A}$ defines a reversible cellular automaton, for example if $\rho_a(a)=\rho_b(b)$ for some $a\in A,b\in A\setminus\{a\}$ then the resulting cellular automaton has $F(a^\omega)=F(b^\omega)$. We refer the reader to \cite{DartnellMaassSchwartz03} for a detailed combinatorial considerations of reversible one-sided cellular automata such that the cellular automaton and its inverse have radius $1$, and to \cite{AcerbiDennunzioFormenti07} for some relations between two-sided cellular automaton $(A^\ZZ,F)$ with one-sided neighborhood and the one-sided cellular automaton $(A^\NN,F)$ defined by the same local rule. For our purposes the example considered in the following Section \ref{021example} is sufficient.

\subsection{Example: A reversible one-sided cellular automaton with positive entropy}
\label{021example}

Define a one-sided cellular automaton $\zot:\{0,1,2\}^\mathbb{N}\rightarrow \{0,1,2\}^\mathbb{N}$ with the following permutations:
\[
\begin{array}{lr}
&0\mapsto 0\\
\rho_0=\rho_2:&1\mapsto 2\\
&2\mapsto 1\\
\end{array}
\qquad
\begin{array}{lr}
&0\mapsto 1\\
\rho_1:&1\mapsto 2\\
&2\mapsto 0\\
\end{array}.
\]
This is reversible, and its inverse also has radius one, as the permutations $\pi_0=\pi_1=(0)(12),\pi_2=(021)$ can be verified to define the inverse of $\zot$. This example was already considered in \cite{DartnellMaassSchwartz03}. We will compute its entropy.

\begin{proposition}
\label{entropy of 012 CA}
The entropy of $(\{0,1,2\}^\NN,\zot)$ is $h(\{0,1,2\}^\NN,\zot)=\frac{1}{2}$.
\end{proposition}
\begin{proof}
According to Proposition \ref{upper-bound-on-entropy} it is enough to compute $p_t(\tau_1(\zot))$. From the local rule we see that $0$ maps to $0$ or $1$, $1$ always maps to $2$, and $2$ maps to $0$ or $1$, and so $\tau_1(\zot)\subseteq\{0,12\}^\mathbb{Z}$ (which is here considered a subshift of $\{0,1,2\}^\mathbb{Z}$). Let $D=\{(0,0),(0,1)\}$ and $x\in st(\zot)$ (recall that in our drawings time, that is $y$-coordinate, increases downwards). If $x_{D},x_{D+(0,2)}\in\{00,12\}$, then $x_{D+(1,1)}=x_D\oplus x_{D+(0,2)}$ where $00$ is interpreted as $0$ and $12$ as $1$, and $\oplus$ is addition modulo two (Figure \ref{fig:201}). Notice also that $201$ can not appear in the trace of $\zot$ since only $\rho_1$ maps $0$ to $1$ but $\pi_1(0)\neq 2$. Together these imply that $\tau_1(\zot)\subseteq\{00,12\}^\ZZ$ since applying the modulo two addition $k$ times to $20^{2k+1}1$ leads to an appearance of $201$ (Figure \ref{fig:odd}). On the other hand we can use the modulo two addition to construct a valid space-time diagram for any $t\in\{00,12\}^\ZZ$ (Figure \ref{fig:xorlike}).

We have seen that $\tau_1(\zot)=\{00,12\}^\mathbb{Z}$, and so $h(A^\NN,\zot)=\frac{1}{2}$.
\end{proof}

Using the direct product construction presented in the end of Section \ref{Preliminaries} to $\zot$ to construct $\zot\times\zot\times\cdots\times\zot$ we can obtain a one-sided reversible cellular automaton that has radius one, whose inverse also has radius one, and that has arbitrarily high entropy.

\begin{figure}
\centering
\begin{minipage}{.30\textwidth}
  \centering
	\begin{tikzpicture}
		\draw[dashed] (0,1) -- ++(1,0);
		\draw[dashed] (1,2) -- ++(1,0);
		\draw[dashed] (0,3) -- ++(1,0);
		\draw (0,0) rectangle node {\large$\boldsymbol b$} +(1,2);
		\draw (0,2) rectangle node {\large $\boldsymbol a$} +(1,2);
		\draw (1,1) rectangle node {\large $\boldsymbol c$} +(1,2);
	\end{tikzpicture}
  \caption{If the pattern formed by $a,b$ and $c$ is a valid pattern in $st(\zot)$ and two of the three $a,b,c$ are in $\{00,12\}$, then so is the third and it is determined by $a\oplus b=c$ where we identify $00$ with $0$ and $12$ with $1$, and $\oplus$ is addition modulo two. This can be verified by simply trying all possible values.}
  \label{fig:201}
\end{minipage}%
\hfill
\begin{minipage}{.30\textwidth}
  \centering
	\begin{tikzpicture}
		\foreach \i/\label in {0/2,1/1,2/0,3/0,4/0,5/0,6/0,7/2,8/1}{
			\draw (0,0.5*\i) rectangle node {$\label$} ++(0.5,0.5);
		}
		\foreach \i/\label in {0/2,1/1,2/0,3/0,4/0,5/2,6/1}{
			\draw (0.5,0.5+0.5*\i) rectangle node {$\label$} ++(0.5,0.5);
		}
		\foreach \i/\label in {0/2,1/1,2/0,3/2,4/1}{
			\draw (1,1+0.5*\i) rectangle node {$\label$} ++(0.5,0.5);
		}
	\end{tikzpicture}
  \caption{The traces of $\zot$ cannot have a word $20^n1$ where $n$ is odd. This can be seen by applying the  modulo two addition as described in Figure \ref{fig:201}, which leads to the word $201$, which is invalid since since only $\rho_1(0)=1$ but $\pi_1(0)\neq 2$.}
  \label{fig:odd}
\end{minipage}%
\hfill
\begin{minipage}{.30\textwidth}
  \centering
	\begin{tikzpicture}
		\foreach \i in {0,1,2,3}{
			\draw (0,\i) rectangle ++(0.5,1);
		}
		\foreach \i in {0,1,2}{
			\draw (0.5,\i+0.5) rectangle ++(0.5,1);
		}
		\foreach \i in {0,1}{
			\draw (1,1+\i) rectangle ++(0.5,1);
		}
		\draw (1.5,1.5) rectangle ++(0.5,1);
		\foreach \i/\label in {0/0,1/0,2/0,3/0,4/2,5/1,6/2,7/1}{
			\node at (0.25,0.5*\i+0.25) {$\label$};
		}
		\foreach \i/\label in {0/0,1/0,2/2,3/1,4/0,5/0}{
			\node at (0.75,0.75+0.5*\i) {$\label$};
		}
		\foreach \i/\label in {0/2,1/1,2/2,3/1}{
			\node at (1.25,1.25+0.5*\i) {$\label$};
		}
		\foreach \i/\label in {0/0,1/0}{
			\node at (1.75,1.75+0.5*\i) {$\label$};
		}
	\end{tikzpicture}
  \caption{Fill the leftmost column in an arbitrary way using the blocks $00$ and $12$. Fill the next column as described in Figure \ref{fig:201}. Notice that we get no violations of the local rule of $F$ doing this. Repeat. }
  \label{fig:xorlike}
\end{minipage}
\end{figure}

It turns out that if we restrict $\zot$ to finite words in $\{0,1,2\}^n$ in such a way that in the last cell, in which we can not use the local rule of $\zot$, we will instead always use $\rho_1$, then this defines a cyclic permutation of $\{0,1,2\}^n$. This gives an unconventional way to enumerate finite words.

\begin{proposition}
\label{enumeration}
Let $A=\{0,1,2\}$ and $(A^\NN,\zot)$ be the cellular automaton as above. Let $\rho$ be the map
\begin{align*}
\bigcup_{n\in\NN\setminus\{0\}}A^n&\stackrel{\rho}{\longrightarrow} \bigcup_{n\in\NN\setminus\{0\}}A^n\\
a_0\cdots a_{n-1}&\longmapsto \rho_{a_1}(a_0)\cdots \rho_1(a_{n-1}).
\end{align*}
Then for every $n\in\NN\setminus\{0\}$ the restriction $\rho:A^n\rightarrow A^n$ is a cyclic permutation.
\end{proposition}
\begin{proof}

Proof goes by induction on $n$: For $n=1$ the claim is clear as $\rho_1$ is a cyclic permutation of $A$.

Suppose that the claim holds for some $n\in\NN\setminus\{0\}$. It is enough to show that the map $x\mapsto\rho^{3^n}(x0^n)_0$ is a cyclic permutation of $A$.

Notice that in $(\rho^i(x0^n))_{i\in\ZZ}$ every occurrence of the pattern in Figure \ref{fig:201} is valid, i.e. does not violate the local rule of $\zot$. Clearly for every $j\in\{1,\dots,n\}$ the configuration $t_j=(\rho^{i}(x0^n)_j)_{i\in\ZZ}$ is periodic. Applying $\rho$ to $0^{n-j+1}$ until we return back to $0^{n-j+1}$ it is straight forward to see that a word of form $0^{n-j+1}t_j'0^{n-j}$, for some finite word $t_j'$, is a smallest period in $t_j$. We claim that the words $t_j'$ cannot contain words $20^k1$ where $k$ is odd. Suppose that such a word does exist. If $k\geq 2(n-j)+1$ then using the argument in Figure \ref{fig:201} we see that this was not the smallest period as we have returned to $0^{n-j+1}$. If $k<2(n-j)+1$ then using the same argument  we reach a contradiction with the observation that the patterns of Figure \ref{fig:201} are valid in $(\rho^i(x0^n))_{i\in\ZZ}$. It is also clear that $t_j'$ starts with a $1$ and ends with a $2$.

Now let $x=0$. Let $t=(\rho^{i}(0^{n+1})_0)_{i\in\{0,1,\dots,3^n-n-1\}}$. By the observations of the previous paragraph we know that the word to the right of $t$ is $0^nt_1'$ where $t_1'\in\{00,12\}^+$. Then $t$ is determined by repeatedly applying the reasoning in Figure \ref{fig:201}. By the induction hypothesis we know that all words of length $n$ appear exactly once in $(\rho^{i}(0^n))_{i\in\{0,1,\dots,3^n-1\}}$. This means that in the right side of $t$ there will be a $1$ exactly $3^{n-1}$ times. In particular this is an odd number of times. Every time we see a $1$ the modulo two addition will swap $t$ between a stream of $00$'s and a stream of $12$'s. Since we start with $00$ and swap for odd number of times, we must end in $12$. So we have that if $n$ is even, then $\rho^{3^n}(00^n)_0=1$, and if $n$ is odd, then $\rho^{3^n}(00^n)_0=2$. It is now enough that we show that if $n$ is even, then $\rho^{3^n}(10^n)_0=2$, and if $n$ is odd, then $\rho^{3^n}(20^n)_0=1$.

Suppose that $n$ is even and $x=1$. The only way the stream of $12$'s in $t$ (defined as before) swaps to a stream of $00$'s is if $2$ in $t$ is aligned with a $1$ in $t_1'$. Before $t_1'$ starts we have $n$ times a $0$ which map $1$ to $2$ and $2$ back to $1$. As we noted, $t_1'\in\{00,12\}^+$ and begins with $1$, so since $n$ is even we see that $2$ is never aligned with $1$ and thus $t$ will only swap between $1$ and $2$ for $3^n$ steps. Since we do this odd number of times, we end up with a $2$.

The case $n$ is odd and $x=2$ goes similarly.



\end{proof}

\section{One-dimensional case}
\label{One-dimensional case}

In \cite{Epperlein17} it was proved that the conjugacy of periodic one-dimensional cellular automata is decidable. They also conjectured that for general one-dimensional cellular automata conjugacy is undecidable. We prove this conjecture. Actually we prove a result that is stronger in a couple of ways: We prove a recursive inseparability result which immediately implies that conjugacy, being a factor, being a subsystem, and the strong variants of all of these are undecidable for both one- and two-sided one-dimensional cellular automata. After the main result we will mention some related problems. These results appeared in the conference paper \cite{JalonenKari18}.

\subsection{Conjugacy of one-dimensional cellular automata}

Our proof relies on the following result.

\begin{theorem}{(\cite{Kari92},\cite{Lewis79})}
\label{nilpotency}
Nilpotency of one-dimensional one-sided cellular automata with a spreading state and radius $1$ is undecidable.
\end{theorem}

We are ready to prove the first main result of this paper.

\begin{theorem}
\label{inseparablePairs}
The following two sets of pairs of one-dimensional one-sided cellular automata are recursively inseparable:
\begin{enumerate}[(i)]
\item pairs where the first cellular automaton has strictly higher entropy than the second one, and
\item pairs that are strongly conjugate and both have zero topological entropy.
\end{enumerate}
\end{theorem}
\begin{proof}
We will reduce the decision problem of Theorem \ref{nilpotency} to this problem, which will prove our claim.

Let $H:B^\mathbb{N}\rightarrow B^\mathbb{N}$ be an arbitrary given one-sided cellular automaton with neighborhood radius $1$ and a spreading quiescent state $q\in B$. Let $k\in\mathbb{N}$ be such that $k>\log_2(\lvert B\rvert)$, $\zot_{2k}$ be the $2k$-fold cartesian product of the cellular automaton $\zot$ of Example \ref{021example}, and $A=\{0,1,2\}^{2k}$. This choice is done to have high enough entropy down the line. Now we are ready to define cellular automata $\mathcal{F}$ and $\mathcal{G}$ such that
\begin{align*}
H\text{ is not nilpotent}&\implies h(\mathcal{F})>h(\mathcal{G})\\
H\text{ is nilpotent}&\implies \mathcal{F}\scong \mathcal{G}\text{ and }h(\mathcal{F})=h(\mathcal{G})=0.
\end{align*}
Both of these new cellular automata work on two tracks $\mathcal{F},\mathcal{G}:\left(A\times B\right)^\mathbb{N}\rightarrow \left(A\times B\right)^\mathbb{N}$. The cellular automaton $\mathcal{G}$ is simply $\id_A\times H$, i.e.
\[
\mathcal{G}((a_0,b_0)(a_1,b_1))=(a_0,H(b_0 b_1)),
\]
for all $a_0,a_1\in A,\;b_0,b_1\in B$. The cellular automaton $\mathcal{F}$ also acts as $H$ on the $B$-track. On the $A$-track $\mathcal{F}$ acts as $\zot_{2k}$ when the $B$-track is not going to become $q$, and as $\id_A$ when the $B$-track is going to become $q$, i.e. 
\[
\mathcal{F}((a_0,b_0)(a_1,b_1))=
\begin{cases}
(\zot_{2k}(a_0a_1),H(b_0b_1)),&\text{if } H(b_0b_1)\neq q\\
(a_0,H(b_0b_1)),&\text{if } H(b_0b_1)=q,
\end{cases}
\]
for all $a_0,a_1\in A,\;b_0,b_1\in B$. 

\textbf{(i)} Suppose that $H$ is not nilpotent. Then the entropy of $\mathcal{G}$ is
\[
h\big(\left(A\times B\right)^\mathbb{N},\mathcal{G}\big)=h\left(A^\mathbb{N},\id_A\right)+h\left(B^\mathbb{N},H\right)=h\left(B^\mathbb{N},H\right),
\]
since $\mathcal{G}=\id_A\times H$. On the other hand, by Proposition \ref{spacetime-diagram-without-spreading-state}, there exists a configuration $e\in B^\mathbb{Z}$ such that for all $i,j\in\mathbb{N}$ we have that $H^i(c)_j\neq q$. But then we have that
\[
h\big(\left(A\times B\right)^\mathbb{N},\mathcal{F}\big)\geq h\left(A^\mathbb{N},\zot_{2k}\right)>\log_2( \lvert B\rvert)\geq h\big(B^\mathbb{N},H\big),
\]
according to Example \ref{021example} and the choice of $k$. Overall we have that
\[
h\big(\left(A\times B\right)^\mathbb{N},\mathcal{F}\big)>h\big(\left(A\times B\right)^\mathbb{N},\mathcal{G}\big),
\]
as was claimed.


\textbf{(ii)} Suppose that $H$ is nilpotent. Let us first explain informally why we now have that $\mathcal{F}\scong \mathcal{G}$. Both $\mathcal{F}$ and $\mathcal{G}$ behave identically on the $B$-track, so the conjugacy will map this track simply by identity. Nilpotency of $H$ guarantees that for all configurations the $B$-track will be $q^\omega$ after some constant time $n$ (Proposition \ref{uniformlyNilpotent}). By the definition of $\mathcal{F}$ this means that after $n$ steps $\mathcal{F}$ does nothing on the $A$-track. Since $\mathcal{G}$ never does anything on the $A$-track, we can use this fact to define the conjugacy on the $A$-track simply with $\mathcal{F}^n$. That this is in fact a conjugacy follows since $\mathcal{F}$ is, informally, reversible on the $A$-layer for a fixed $B$-layer.

Let us be exact. First we will define a continuous map $\phi:(A\times B)^\mathbb{N}\rightarrow (A\times B)^\mathbb{N}$ such that $\phi \mathcal{F} =\mathcal{G} \phi$. This $\phi$ will be a cellular automaton. Then we show that $\phi$ is injective, which implies reversibility (by \cite{Hedlund69}), and so we will have $((A\times B)^\NN,\mathcal{F})\scong ((A\times B)^\NN,\mathcal{G})$.

Let $\pi_A:A^\mathbb{N}\times B^\mathbb{N}\rightarrow A^\mathbb{N}$ be the projection $\pi_A(c,e)=c$ for all $c\in A^\mathbb{N}$ and $e\in B^\mathbb{N}$. Define $\pi_B:A^\mathbb{N}\times B^\mathbb{N}\rightarrow B^\mathbb{N}$ similarly.

Let $n\in\mathbb{N}$ be a number such that for all $c\in B^\mathbb{N}$ we have $H^n(c)=q^\omega$. Such $n$ exists according to Proposition \ref{uniformlyNilpotent}, since $H$ is nilpotent. Because $\mathcal{F}$ and $\mathcal{G}$ act identically on the $B$-track, $\phi$ will map this layer simply by identity, i.e.
\[
\pi_B \phi(c,e)=e,
\]
for all $c\in A^\mathbb{N},e\in B^\mathbb{N}$. After $n$ steps $\mathcal{F}$ does nothing on the $A$-track, i.e. acts the same way as $\mathcal{G}$ does. Because of this we define
\[
\pi_A\phi = \pi_A {\mathcal{F}}^n.
\]
Now $\phi$ is a cellular automaton, since it is continuous and shift-commuting. Let us show that $\phi$ is a homomorphism. Of course we have that
\[
\phi \mathcal{F} = \mathcal{G} \phi \iff \big(\pi_A\phi \mathcal{F} = \pi_A \mathcal{G}\phi\text{ and } \pi_B\phi \mathcal{F} = \pi_B \mathcal{G}\phi\big).
\]
It is immediate from the definitions that $\pi_B\phi \mathcal{F} = \pi_B \mathcal{G}\phi$. For the equality on the $A$-layer notice first that $\pi_A\mathcal{G}=\pi_A$, and then compute:
\begin{align*}
\pi_A\phi \mathcal{F} \;&\stackrel{\mathmakebox[\widthof{=}]{\text{def.}}}{=}\; (\pi_A {\mathcal{F}}^n)\mathcal{F}\\
&=\;\pi_A \mathcal{F}{\mathcal{F}}^n & &\mid\mid\text{ after }n\text{ steps } \mathcal{F}\\
&=\;\pi_A \mathcal{G} {\mathcal{F}}^n & &\quad\;\;\text{behaves as }\mathcal{G}\\
&=\;\pi_A {\mathcal{F}}^n\\
&\stackrel{\mathmakebox[\widthof{=}]{\text{def.}}}{=}\;\pi_A\phi\\
&=\;\pi_A \mathcal{G}\phi.
\end{align*}
So we have that $\phi \mathcal{F} = \mathcal{G} \phi$.

To prove that $\phi$ is a strong \emph{conjugacy} it is enough to show that $\phi$ is an injection. As the $B$-layer is mapped by identity, we only need to show that for a fixed $e\in B^\mathbb{N}$ we have that for all $c\in A^\mathbb{N}$ there exists a unique $c'\in A^\mathbb{N}$ such that $\phi(c',e)=(c,e)$. By the definition of $\phi$ this will hold if
\begin{align*}
\pi_A {\mathcal{F}}^n(\_,e):A^\mathbb{N}&\longrightarrow A^\mathbb{N}\\
c\;\;&\longmapsto \pi_A {\mathcal{F}}^n(c,e)
\end{align*}
is a bijection for every $e\in B^\mathbb{N}$. We can consider this step by step. We claim that
\[
(c,e)=(c_0c_1c_2\dots,e_0e_1e_2\dots)\in(A\times B)^\mathbb{N}
\]
uniquely defines the $A$-track of the elements in the set $\mathcal{F}^{-1}(c,e)$. Let $(c',e')=(c_0'c_1'c_2'\dots,e_0'e_1'e_2'\dots)\in\mathcal{F}^{-1}(c,e)$. It is enough to show that $c_0'$ is defined uniquely by $(c,e)$. Suppose first that $e_0=q$. Then according to the definition $\mathcal{F}$ acted as identity, so we have that $c_0'=c_0$. Suppose next that $e_0\neq q$. We have two cases: either $e_1=q$ or not. Suppose first that $e_1=q$. Then as before we have that $c_1'=c_1$. And so $c_0'=\rho_{c_1'}^{-1}(c_0)=\rho_{c_1}^{-1}(c_0)$. And lastly suppose that $e_1\neq q$. Then we have that $\zot_{2n}(c_0'c_1'c_2'\dots)=(c_0c_1\dots)$ according to the definition of $\mathcal{F}$. But now $c_0'$ is uniquely determined since $\zot_{2k}$ is reversible and the inverse also has radius $1$, namely we have that $c_0'=\zot_{2k}^{-1}(c_0c_1)$.

To complete the proof we observe that
\[
h(\mathcal{F})=h(\mathcal{G})=h(\id_A)+h(H)=0,
\]
since $\mathcal{F}\scong \mathcal{G}=\id_A\times H$, and $H$ is nilpotent.

\end{proof}

Since Theorem \ref{inseparablePairs} can be reduced to the two-sided case, also the two-sided variant is undecidable. We also get the following corollary.

\begin{corollary}
\label{undecidableProperties}
Let $\mathbb{M}=\mathbb{N}$ or $\mathbb{M}=\mathbb{Z}$. Let $F,G:A^\mathbb{M}\rightarrow A^\mathbb{M}$ be two given cellular automata. Then the following hold:
\begin{enumerate}
\item It is undecidable whether $F$ and $G$ are (strongly) conjugate.
\item It is undecidable whether $F$ is a (strong) factor of $G$.
\item It is undecidable whether $F$ is a (strong) subsystem of $G$.
\end{enumerate}
\end{corollary}

\begin{proof}
{\it 1.} The pairs in the set {\it (i)} of Theorem \ref{inseparablePairs} cannot be (strongly) conjugate, and the pairs in {\it(ii)} have to be. Thus deciding (strong) conjugacy would separate these sets.

{\it 2.} One of the cellular automata in the pair from the set {\it (i)} has strictly higher entropy than the other, so it cannot be a (strong) factor of the other. On the other hand cellular automata of pairs from the set {\it (ii)} are (strong) factors of each other. So checking whether both cellular automata of a pair is a (strong) factor of the other would separate the sets of Theorem \ref{inseparablePairs}.

{\it 3.} In a similar way, since a subsystem cannot have higher entropy.

\end{proof}

\subsection{Restricted cases}

We have seen that (strong) conjugacy is undecidable in general for one-dimensional cellular automata. A natural follow-up question is whether (strong) conjugacy remains undecidable even if we restrict to some natural subsets of cellular automata. When restricted to periodic cellular automata, the following is known.

\begin{theorem}{(\cite[Corollary 5.17.]{Epperlein17})}
Conjugacy of periodic cellular automata on one- or two-sided subshifts of finite type is decidable.
\end{theorem}

Periodic cellular automata are the least sensitive to changes in the initial configuration, and are precisely the reversible equicontinuous cellular automata. Naturally one could ask what happens if the requirement of reversibility is dropped, i.e. is conjugacy of eventually periodic cellular automata decidable (\cite[Question 8.6.]{Epperlein17}), or the strong conjugacy of either.

In the other end of the sensitivity scale are the cellular automata that are the most sensitive to initial conditions, i.e. positively expansive ones. A cellular automaton $(A^\ZZ,F)$ is called \emph{positively expansive} if
\[
\exists\varepsilon>0:\forall c,e\in A^\ZZ: c\neq e\implies (\exists n\in\mathbb{N}:d(F^n(c),F^n(e))>\varepsilon). 
\]
Positively expansive cellular automata are quite extensively studied which allows us to deduce the following result.

\begin{proposition}
Conjugacy of positively expansive cellular automata on one- or two-sided full shifts is decidable.
\end{proposition}
\begin{proof}
Let $F:A^\mathbb{M}\rightarrow A^\mathbb{M}$ and $G:B^\mathbb{M}\rightarrow B^\mathbb{M}$ be two positively expansive cellular automata. Due to the positive expansivity, $F$ and $G$ are conjugate to $\tau_k(F)$ and $\tau_k(G)$ (resp.) for large enough $k$. These subshifts are conjugate to subshifts of finite type (\cite{BoyleKitchens99} for one-sided case, \cite{Nasu95} for two-sided case). According to \cite[Theorem 36]{diLena07} we can effectively compute these subshifts. The claim follows, as the conjugacy of one-sided subshifts of finite type is decidable (\cite{Williams73}).
\end{proof}

Again it is natural to ask whether strong conjugacy of positively expansive cellular automata is decidable or (strong) conjugacy of expansive cellular automata, i.e. reversible cellular automata $(A^\ZZ,F)$ such that
\[
\exists\varepsilon>0:\forall c,e\in A^\ZZ: c\neq e\implies(\exists n\in\mathbb{Z}: d(F^n(c),F^n(e))>\varepsilon).
\]
It is conjectured, and known in some special cases, that expansive cellular automata are conjugate to two-sided subshifts of finite type. However, even if this conjecture holds the above reasoning could not be used as it is, since it is not known whether conjugacy of two-sided subshifts of finite type is decidable.


\subsection{Conjugacy of subshifts}

Let $X\subseteq A^\MM,Y\subseteq B^\MM$ be two subshifts, i.e. topologically closed and shift invariant subsets. The shift dynamical systems $(X,\sigma)$ and $(Y,\sigma)$ are conjugate if there exists a homeomorphism $\phi:X\rightarrow Y$ such that $\phi \sigma=\sigma\phi$. Perhaps the most important open questions in symbolic dynamics consider conjugacy of subshifts. Conjugacy of one-sided subshifts of finite type is known to be decidable (we used this in the previous section), but the same problem for two-sided subshifts of finite type and for one- and two-sided sofic shifts is unknown.

From the undecidability of strong conjugacy we get the following undecidability result regrading (even one-sided) subshifts of finite type.



\begin{proposition}
Let $X,Y\subseteq(A\times A)^\mathbb{M}$ be two subshifts of finite type. It is undecidable whether $X$ and $Y$ are conjugate via a conjugacy of the form $\phi\times\phi$.
\end{proposition}
\begin{proof}
The proof is a direct reduction from the undecidability of strong conjugacy of cellular automata. Let $F,G:A^\mathbb{M}\rightarrow A^\mathbb{M}$ be two cellular automata. Let $X=\{(c,F(c))\mid c\in A^\mathbb{M}\}$ and $Y=\{(c,G(c))\mid c\in A^\mathbb{M}\}$. These subshifts are naturally conjugate to $A^\mathbb{M}$. Suppose there exists a conjugacy $\phi\times\phi:X\rightarrow Y$. Then $\phi$ commutes with the shift and for every $c\in A^\mathbb{M}$ we have that $(\phi(c),\phi F(c))=(e,G(e))$, where  $e$ has to be $\phi(c)$, and so $\phi F(c)=G\phi(c)$ for all $c\in A^\mathbb{M}$. In other words $\phi$ is a strong conjugacy of $(A^\mathbb{M},F)$ and $(A^\mathbb{M},G)$.

On the other hand, any strong conjugacy $\phi$ from $(A^\mathbb{M},F)$ to $(A^\mathbb{M},G)$ immediately gives a conjugacy $\phi\times\phi$ between $X$ and $Y$.
\end{proof}


\section{Two-dimensional case}
\label{Two-dimensional case}

After seeing that conjugacy is undecidable for cellular automata, one immediate question is to ask whether this holds for reversible cellular automata. In this section we turn to two-dimensional cellular automata and show that for two-dimensional reversible cellular automata conjugacy, being a factor, being a subsystem, and the strong variants of all of these are undecidable properties. The fact that conjugacy is undecidable was already proved in \cite{Epperlein17} even when restricted to periodic cellular automata with period two.

Notice that in this section we will only consider two-sided cellular automata, and will simply call them cellular automata. There has not been much study on the one-sided cellular automata beyond the one-dimensional case. Our proof borrows a lot from the proof that reversibility is undecidable for two-dimensional cellular automata, but this result is not known for one-sided two-dimensional cellular automata. Notice however that for example surjectivity is undecidable also for one-sided two-dimensional cellular automata as this follows from the two-sided case.

\subsection{Conjugacy of reversible two-dimensional cellular automata}

Recall that we denote $\mathcal{C}_n=[0,n)^2$. Let $\mathcal{A}\subseteq A^{\mathcal{C}_n}$ be a set of patterns, considered here to be valid, and define a \emph{direction function} $\delta:A\rightarrow\{(\pm 1,0),(0,\pm 1)\}$. A sequence $p_1,p_2,\dots, p_k\in\ZZ^2$ is a \emph{$\delta$-path} on $c\in A^{\ZZ^2}$ if $p_{i+1}=p_i+\delta(c_{p_i})$ for all $i\in\{1,\cdots,k-1\}$. A $\delta$-path $p_1,\dots, p_k$ is \emph{($\mathcal{A}$-)valid} if for every $i\in\{1,\dots,k\}$ we have that $\sigma_{p_i}(c)_{\mathcal{C}_n}\in \mathcal{A}$ and for all $(x,y)\in\ZZ^2$ we have that $p_i+\delta(c_{p_i})=(x,y)+\delta(c_{(x,y)})\implies (x,y)=p_i$. The first condition says that the pattern we see at $p_i$ is valid, and the second condition says that a valid path does not branch when tracing the path backwards. A pair $(\mathcal{A},\delta)$ is an \emph{orientation} on the full shift $A^{\ZZ^2}$. An orientation $(\mathcal{A},\delta)$ is \emph{acyclic} if every $\mathcal{A}$-valid $\delta$-path contains no cycles. Let $c\in A^{\ZZ^2}$ and $p\in\ZZ^2$ a position which is part of a valid path in $c$. Then $p$ is the \emph{beginning} of the valid path if $q,p$ is not a valid path for any $q\in\ZZ^2$. Similarly $p$ is the \emph{end} of the valid path if $p,q$ is not a valid path for any $q\in\ZZ^2$. Lastly $p$ is \emph{in the middle} of the valid path if it is not the end or the beginning of the path. Notice that being an end, a beginning, or in the middle of a valid path is a local property. Directly from \cite{Kari94} we get the following result. 
 
\begin{proposition}[\cite{Kari94}]
\label{infinitepaths}
Given an acyclic orientation $(\mathcal{A},\delta)$ on $A^{\ZZ^2}$, it is undecidable whether there exists an infinite $\mathcal{A}$-valid $\delta$-path. 
\end{proposition}
Notice that by compactness we have that if all valid paths are finite, then there is a global bound on the length of the valid paths.

Using similar notations as in \cite{Meyerovitch08} we write $I(c)$ for the maximal number of pairwise disjoint infinite valid paths in $c$, and for an orientation $(\mathcal{A},\delta)$ we denote $I(\mathcal{A},\delta)=\sup_{c\in A^{\ZZ^2}} I(c)$. Considerations of \cite[Section 4]{Meyerovitch08} say that only a bounded number of infinite valid paths that the acyclic orientation in \cite{Kari94} defines can fit in any one configuration. Combining this with Proposition \ref{infinitepaths} above we get the following.

\begin{proposition}
\label{infinitepaths2}
Given an acyclic orientation $(\mathcal{A},\delta)$ on $A^{\ZZ^2}$ such that $I(\mathcal{A},\delta)<\infty$, it is undecidable whether $I(\mathcal{A},\delta)=0$ or not.
\end{proposition}

Now we proceed to defining cellular automata by adding a layer on top of $A^{\ZZ^2}$ in a similar fashion as in \cite{Kari94} and \cite{Meyerovitch08}. For an acyclic orientation $(\mathcal{A},\delta)$ with $I(\mathcal{A},\delta)<\infty$ we want to define two reversible cellular automata such that if $I(\mathcal{A},\delta)=0$ then the cellular automata are strongly conjugate and have zero entropy, and if $I(\mathcal{A},\delta)\neq 0$, then one of the cellular automata has strictly larger entropy than the other one.

On top of $A^{\ZZ^2}$ we put another layer $B^{\ZZ^2}$ on which we will simulate one-dimensional cellular automata $B^{\ZZ}\rightarrow B^{\ZZ}$ on the valid paths of $A^{\ZZ^2}$. We want to simulate the shift map $\sigma$. However the simple shift map alone leads to non-reversible cellular automata on finite valid paths as information is either lost or has to be made up at the beginnings and ends of valid paths. To avoid this we will take $B$ to be $B_1\times B_1$ for some finite set $B_1$ and the map $\widehat{\sigma}:B^\ZZ\rightarrow B^\ZZ$ that shifts the first track to the left and the second one to the right, i.e. $\widehat{\sigma}(c,e)_i=(c_{i+1},e_{i-1})$ for all $(c,e)\in B^\ZZ=(B_1\times B_1)^\ZZ$. In the beginnings and ends of valid paths we will simply move the content from one track to the other forming a cycle. For technical reasons we will further take that $B_1=B_2\times B_2$ for a finite set $B_2$ (with at least two elements). Now our one-dimensional cellular automaton is $\widehat{\sigma}:(B_2^4)^\ZZ\rightarrow (B_2^4)^\ZZ$ but this should be considered as one two-track tape going to left and right. The choice $B_1=B_2\times B_2$ is done so that we can give two different ways to restrict $\widehat{\sigma}$ to the finite valid paths.

\begin{figure}
\centering
\begin{minipage}{.45\textwidth}
  \centering
\begin{tikzpicture}
\foreach \i in {0,1,2}{
	\draw (0.75*\i,0) rectangle node {$v_\i$} ++(0.75,0.75);
	\draw (0.75*\i,0.75) rectangle node {$u_\i$} ++(0.75,0.75);
}
\draw (2.25,0) rectangle ++(1.5,1.5);
\draw (3.75,0) rectangle node {$v_{n-1}$} ++(0.75,0.75);
\draw (3.75,0.75) rectangle node {$u_{n-1}$} ++(0.75,0.75);
\draw[line width = 0.3mm, dotted] (2.5,0.75) -- (3.5,0.75);
\draw[->, line width=0.3mm] (0,1.125) arc (90:270:0.7) -- (4.5,-0.275) arc (-90:90:0.7);
\draw[->, line width=0.3mm] (0,0.375) arc (90:270:0.25625) -- (4.5,-0.1375) arc (-90:90:0.25625);
\end{tikzpicture}
  \caption{Illustration of $\shift$.}
  \label{fig:shift}
\end{minipage}%
\hfill
\begin{minipage}{.45\textwidth}
  \centering
\begin{tikzpicture}
\foreach \i in {0,1,2}{
	\draw (0.75*\i,0) rectangle node {$v_\i$} ++(0.75,0.75);
	\draw (0.75*\i,0.75) rectangle node {$u_\i$} ++(0.75,0.75);
}
\draw (2.25,0) rectangle ++(1.5,1.5);
\draw (3.75,0) rectangle node {$v_{n-1}$} ++(0.75,0.75);
\draw (3.75,0.75) rectangle node {$u_{n-1}$} ++(0.75,0.75);
\draw[line width = 0.3mm, dotted] (2.5,0.75) -- (3.5,0.75);
\draw[->, line width=0.3mm] (0,1.125) arc (90:270:0.7) -- (4.5,-0.275) arc (-90:90:0.325);
\draw[->, line width=0.3mm] (0,0.375) arc (90:270:0.25625) -- (4.5,-0.1375) arc (-90:90:0.63125);
\end{tikzpicture}
  \caption{Illustration of $\shiftt$.}
  \label{fig:shiftt}
\end{minipage}%
\end{figure}

We define two maps, $\shift$ and $\shiftt$, on finite words $(B_2\times B_2)^+$ as follows:
\[
\shift \left(
\begin{array}{ccccc}
	u_0 & u_1 & \cdots & u_{n-2} & u_{n-1} \\
	v_0 & v_1 & \cdots & v_{n-2} & v_{n-1}
\end{array} \right)
=
\left(
\begin{array}{ccccc}
	u_1 & u_2 & \cdots & u_{n-1} & u_0\\
	v_1 & v_2 & \cdots & v_{n-1} & v_0
\end{array} \right),
\]
and $\shiftt$ as
\[
\shiftt \left(
\begin{array}{ccccc}
	u_0 & u_1 & \cdots & u_{n-2} & u_{n-1} \\
	v_0 & v_1 & \cdots & v_{n-2} & v_{n-1}
\end{array} \right)
=
\left(
\begin{array}{ccccc}
	u_1 & u_2 & \cdots & u_{n-1} & v_0\\
	v_1 & v_2 & \cdots & v_{n-1} & u_0
\end{array} \right),
\]
where ${u_i\atop v_i}\in B_2\times B_2$ for all $i\in\{0,1,\dots,n-1\}$. The map $\shift$ is obtained by taking a finite word, gluing the ends together, and applying the shift map $\sigma:(B_2\times B_2)^\ZZ\rightarrow (B_2\times B_2)^\ZZ$ locally (Figure \ref{fig:shift}). One can also consider $\shiftt$ to be obtained from $\sigma$ by gluing the ends of finite words together, but this time the tape is also flipped to form a M\"obius strip (Figure \ref{fig:shiftt}). Notice that if we restrict $\shift$ and $\shiftt$ to the words of even length $2n$, then we have a bijection $\phi$ such that $\shift=\phi^{-1}\shiftt^2\phi$, namely
\begin{equation}
\label{phi}
\phi \left(
\begin{array}{ccccc}
	u_0 & u_1 & \cdots & u_{2n-2} & u_{2n-1} \\
	v_0 & v_1 & \cdots & v_{2n-2} & v_{2n-1}
\end{array} \right)
=
\left(
\begin{array}{ccccc}
	u_0 & v_0 & \cdots & u_{n-1} & v_{n-1}\\
	u_n & v_n & \cdots & u_{2n-1} & v_{2n-1}
\end{array} \right).
\end{equation}

Let $(\mathcal{A},\delta)$ be an acyclic orientation of $A^{\ZZ^2}$ with $I(\mathcal{A},\delta)<\infty$. We will define two cellular automata $F_{\shift},F_{\shiftt}:(A\times B)^{\ZZ^2}\rightarrow(A\times B)^{\ZZ^2}$ where $B=(B_2\times B_2)^2$ as was defined above. Both will map the $A$-layer by identity. On the $B$-layer we use $\shift$ for $F_\shift$ and $\shiftt$ for $F_{\shiftt}$. To be more exact: Let $c\in A^{\ZZ^2},e\in B^{\ZZ^2},p_1\in\ZZ^2$, and $e_{p_1}=(a_1,b_1,x_1,y_1)$. We will define $F_\shift$ and $F_\shiftt$ in cases:

\noindent$\bullet$ \underline{If $p_1$ is not part of a valid path in $c$}, then
\[
F_\shift(c,e)_{p_1}=F_\shiftt(c,e)_{p_1}=(c,e)_{p_1}.
\]
\noindent$\bullet$ \underline{If $p_1$ is a beginning of a valid path}, and there exists $p_2\in\ZZ^2$ such that $p_1,p_2$  is valid (so that $p_1$ is not also an end), and let $e_{p_2}=(a_2,b_2,x_2,y_2)$, then
\[
F_\shift(c,e)_{p_1}=F_\shiftt(c,e)_{p_1}=(c_{p_1},(a_2,b_2,a_1,b_1)).
\]
\noindent$\bullet$ \underline{If $p_1$ is in the middle of a valid path}, say $p_0,p_1,p_2$ is valid, $e_{p_0}=(a_0,b_0,x_0,y_0)$, and $e_{p_2}=(a_2,b_2,x_2,y_2)$, then
\[
F_\shift(c,e)_{p_1}=F_\shiftt(c,e)_{p_1}=(c_{p_1},(a_2,b_2,x_0,y_0)).
\]
We are left with the cases when $F_\shift$ and $F_\shiftt$ behave differently, namely at the ends of valid paths. 

\noindent$\bullet$ \underline{If $p_1$ is an end of a valid path}, $p_0\in\ZZ$ such that $p_0,p_1$ is valid, and $e_{p_0}=(a_0,b_0,x_0,y_0)$, then
\[
F_\shift(c,e)_{p_1}=(c_{p_1},(x_1,y_1,x_0,y_0))\qquad\text{and}\qquad F_\shiftt(c,e)_{p_1}=(c_{p_1},(y_1,x_1,x_0,y_0)).
\]
\noindent$\bullet$ \underline{If $p_1$ is both the beginning and the end of a valid path}, then
\[
F_\shift(c,e)_{p_1}=(c_{p_1},(x_1,y_1,a_1,b_1))\qquad\text{and}\qquad F_\shiftt(c,e)_{p_1}=(c_{p_1},(y_1,x_1,a_1,b_1)).
\]
All this is to say that $F_\shift$ and $F_\shiftt$ simulate $\shift$ and $\shiftt$ (resp.) on the valid paths. Notice that $F_\shiftt^2$ simulates $\shiftt^2$ on the valid paths, so it is natural to define $F_{\shiftt^2}=F_\shiftt^2$.

Suppose that $0<I(\mathcal{A},\delta)<\infty$. By the reasoning of \cite[Lemma 3.2., Lemma 3.3., Theorem 3.4.]{Meyerovitch08} we get that there is a confiugration $c\in A^{\ZZ^2}$ such that the entropy of $((A\times B)^{\ZZ^2},F_\shift)$ and $(A\times B)^{\ZZ^2},F_{\shiftt^2})$ is achieved even when restricting the $A$-layer to $c$. With this fixed background the entropy depends only on the one-dimensional cellular automata which are simulated on the valid paths. Then $F_\shift$ and $F_{\shiftt^2}$ cannot be (strongly) conjugate, since the latter has higher entropy.

Next suppose that $I(\mathcal{A},\delta)=0$, so that there can only be finite valid paths. Then by compactness we have a global bound $M\in\NN$ such that for any valid path $p_1,p_2,\dots,p_k$ holds that $k<M$. Of course there then also exists $m\in\NN$ such that any valid path fits inside a suitably positioned $\mathcal{C}_m$. We will define a strong conjugacy $H_\phi:(A\times B)^{\ZZ^2}\rightarrow (A\times B)^{\ZZ^2}$ of $((A\times B)^{\ZZ^2},F_\shift)$ and $((A\times B)^{\ZZ^2},F_{\shiftt^2})$ based on the map $\phi$ defined by (\ref{phi}) above. The local rule of $F_\phi$ has domain $(A\times B)^{[-m,m+n]^2}$ where $n$ is such that $\mathcal{A}\subseteq A^{\mathcal{C}_n}$. This domain guarantees that for any $c\in(A\times B)^{\ZZ^2}$ and $p\in \ZZ^2$ we can recognize the entire valid path that $p$ is part of. Let $(c,e)\in(A\times B)^{\ZZ^2}$. We define $H_\phi(c,e)_p$ for an arbitrary $p\in\ZZ^2$. If $p$ is not part of a valid path, then $H_\phi(c,e)_p=(c,e)_p$. Suppose $p$ is part of a valid path and let $p_1,p_2,\dots p_k$ be the valid path such that $p_1$ is the beginning of the path, $p_k$ the end of the path, and $p=p_i$ for some $i\in\{1,\dots, k\}$. As pointed out, the local neighborhood is large enough so that the local rule sees this entire valid path and can verify its validity on each position of the valid path. Denote $e_{p_j}=(a_j,b_j,x_j,y_j)$ for all $j\in\{1,\dots,k\}$. Now we define
\begin{align*}
H_\phi(c,e)_p = \Bigg(c_p,\bigg( \phi \left(
\begin{array}{cccccc}
	x_k & \cdots & x_1 & a_1 & \cdots & a_k \\
	y_k & \cdots & y_1 & b_1 & \cdots & b_k
\end{array} \right)_{k+i-1},\\
\phi \left(
\begin{array}{cccccc}
	x_k & \cdots & x_1 & a_1 & \cdots & a_k \\
	y_k & \cdots & y_1 & b_1 & \cdots & b_k
\end{array} \right)_{i-1}
\bigg)\Bigg).
\end{align*}
Since $\phi$ is a bijection on words of even length we get that $((A\times B)^{\ZZ^2},F_\shift)\scong ((A\times B)^{\ZZ^2},F_{\shiftt^2})$. Since $F_\shift$ and $F_{\shiftt^2}$ are now periodic, they have zero entropy. Overall we have seen the following.

\begin{theorem}
\label{inseparablePairs2d}
The following two sets of pairs of reversible two-dimensional cellular automata are recursively inseparable:
\begin{enumerate}[(i)]
\item pairs where the first cellular automaton has strictly higher entropy than the second one, and
\item pairs that are strongly conjugate and both have zero entropy.
\end{enumerate}
\end{theorem}

The same way we got Corollary \ref{undecidableProperties} from Theorem \ref{inseparablePairs}, we now get the following corollary.

\begin{corollary}
\label{undecidableProperties2d}
Let $F,G:A^{\ZZ^2}\rightarrow A^{\ZZ^2}$ be two reversible cellular automata. Then the following hold:
\begin{enumerate}
\item It is undecidable whether $F$ and $G$ are (strongly) conjugate.
\item It is undecidable whether $F$ is a (strong) factor of $G$.
\item It is undecidable whether $F$ is a (strong) subsystem of $G$.
\end{enumerate}
\end{corollary}



\begin{remark}
Here is an alternative construction which uses one-sided reversible cellular automata, which allows using the construction and arguments of \cite{Meyerovitch08} more directly. Let $\zot:\{0,1,2\}^\NN\rightarrow\{0,1,2\}^\NN$ as in Section \ref{021example}, and $A^{\ZZ^2}$ a full shift with an orientation $(\mathcal{A},\delta)$. We will simulate $\zot$ on valid paths as we did $\shift$ and $\shiftt$. Since $\zot$ is one-sided, we do not need multiple tracks in this construction. More precisely $F_\zot:(A\times\{0,1,2\})^{\ZZ^2}\rightarrow (A\times\{0,1,2\})^{\ZZ^2}$ is defined by $F_\zot((c,e))_p=(c_p,\zot(e_pe_{p+\delta(c_{p})}))$ on beginnings and middle points of valid paths, and by identity on invalid positions. We still have to define $F_\zot$ on the ends of valid paths: the value at the end of a valid path will be permuted by $\rho_1$. In other words, $F_\zot$ uses the map $\rho$ from Proposition \ref{enumeration} on finite words.

Suppose there exists only finite valid paths. Then there is an upper bound $n\in\NN$ for the length of the valid paths. By Proposition \ref{enumeration} we know that on a finite path of length $k\leq n$ the cellular automaton $F_\zot$ enumerates all finite words in $\{0,1,2\}^k$. The same holds also for $F_\zot^2$ since $3^n$ is odd. Now we can define a conjugacy $\phi$ which fixes the $A^{\ZZ^2}$-layer and all letters on invalid paths on the $\{0,1,2\}^{\ZZ^2}$-layer, and, slightly informally, on valid paths of length $k$ we define $\phi:\{0,1,2\}^k\rightarrow \{0,1,2\}^k$ by setting $\phi(\rho^i(0^k))=\rho^{2i}(0^k)$  for all $i\in\{0,1,\dots,3^k-1\}$. This shows that $F_\zot$ and $F_\zot^2$ are conjugate.

Suppose there exists infinite valid paths. Then $F_\zot^2$ has larger entropy than $F_\zot$ since $\zot^2$ has higher entropy than $\zot$. This concludes the proof
\end{remark}

From this remark we get the following variant.

\begin{corollary}
The following sets of two-dimensional reversible cellular automata are recursively inseparable:
\begin{enumerate}[(i)]
\item cellular automata $F:A^{\ZZ^2}\rightarrow A^{\ZZ^2}$ such that $F^2$ has strictly higher entropy than $F$, and
\item cellular automata $F:A^{\ZZ^2}\rightarrow A^{\ZZ^2}$ such that $F\scong F^2$, and $h(A^{\ZZ^2},F)=0$.
\end{enumerate}
\end{corollary}

Notice that Epperlein's result \cite[Corollary 5.19.]{Epperlein17} says that conjugacy is undecidable even among two-periodic cellular automata, while all our results restrict only to reversible cellular automata. Strengthening the undecidability of strong conjugacy to periodic cellular automata seems plausible using the cellular automata from \cite[Example 7.6.]{Epperlein17} in our construction instead of $\shift$ and $\shiftt$. Example \cite[Example 7.6.]{Epperlein17} presents two one-dimensional cellular automata which are (temporally) periodic and conjugate on (spatially) periodic configurations but not conjugate in general, and thus not strongly conjugate either. Using these we would still have that if all valid paths are finite then the constructed cellular automata are conjugate. However the entropy argument does not work in the case that also infinite valid paths exist, since the entropy of a periodic cellular automaton is zero. It is not clear that even though the one-dimensional cellular automata simulated are not conjugate that the two-dimensional cellular automata could not be.

Lastly we note that in \cite[Question 6.1.]{Meyerovitch08} Meyerovitch asked whether for $d>1$ there exists an injective $d$-dimensional cellular automaton which has finite non-zero entropy. Either of the constructions given above explicitly gives a positive answer to this question for $d=2$. Simulating the one-dimensional cellular automata presented here on Meyerovitch's $d$-dimensional oriented full shifts, one also gets a positive answer for any larger $d$.

\begin{proposition}
For any $d\in\NN\setminus\{0\}$ there exists a reversible cellular automaton $F:A^{\ZZ^d}\rightarrow A^{\ZZ^d}$ such that $0<h(A^{\ZZ^d},F)<\infty$.
\end{proposition}





\section{Conclusion}

We proved that the following sets of pairs of cellular automata are recursively inseparable
\begin{enumerate}[(i)]
\item cellular automata $F,G:A^\MM\rightarrow A^\MM$ such that $F$ has strictly higher entropy than $G$ , and
\item cellular automata $F,G:A^\MM\rightarrow A^\MM$ such that $(A^\MM,F)\scong (A^\MM,G)$, and both have zero entropy.
\end{enumerate}
This implies that the decision problems ``are (strongly) conjugate'', ``is a (strong) subsystem of'' and ``is a (strong) factor of'' are undecidable for one-dimensional one- and two-sided cellular automata. For two-dimensional two-sided cellular automata we also proved the same inseparability when restricted to reversible cellular automata. Naturally one is led to the following problem:

\begin{question}
Is (strong) conjugacy undecidable for reversible two-sided one-dimensional cellular automata? What about for reversible one-sided one-di\-men\-si\-o\-nal cellular automata?
\end{question}

In the reversible case one- and two-sided variants seem more distant from each other. In the two-sided reversible case it is known for example that periodicity and left-expansivity are undecidable propreties (\cite{KariOllinger08} and \cite{KariLukkarila09} resp.). These existing undecidability provide possible replacements of the nilpotency problem used in our construction. However there is a lot less to work with in the reversible one-sided case; so far the only undecidability result known is that if one is given a reversible cellular automaton $(A^\NN,F)$ and a configuration $x\in A^\NN$ with a simple description, then it is undecidable whether $x$ is $F$-periodic or not \cite{DelacourtOllinger17}. Since this result considers a given configuration rather than the dynamics of $F$, it seems unfit for the conjugacy problem.

Here we considered problems in a topological dynamical setting. One can also consider algebraic variants of these decision problems by fixing the underlying alphabet as cellular automata over a fixed alphabet form a monoid and reversible cellular automata over a fixed alphabet form a group. This additional restriction can indeed be added, but it does require some small additional considerations. For example, the decision problems used in our reductions need to be replaced with fixed size alphabet variants (see \cite[Proposition 2.4.]{DurandFormentiVarouchas03} for fixed alphabet variant of the undecidability of nilpotency).

Lastly we note that according to \cite[Theorem 5.18.]{Epperlein17} conjugacy remains undecidable for reversible one-sided two-dimensional cellular automata, that is cellular automata $F:A^{\NN^2}\rightarrow A^{\NN^2}$. Looking for a strong variant we face the same difficulties as for one-sided one-dimensional cellular automata in the lack of existing undecidability results. Our proof that strong conjugacy is undecidable for reversible two-dimensional cellular automata relied on the construction used to prove that reversibility is undecidable. In many cases questions about one-sided cellular automata can directly be answered using the known results about two-sided cellular automata, since two-sided cellular automaton composed with suitable shift map gives a one-sided cellular automaton, and many properties are preserved this way. This is why it is easy to see that surjectivity remains undecidable for one-sided two-dimensional cellular automata. However this cannot be used with reversibility, since not all reversible two-sided cellular automata can be shifted to give a \emph{reversible} one-sided cellular automaton. 

\begin{question}
Is it decidable whether a given two-dimensional one-sided cellular automaton $F:A^{\NN^2}\rightarrow A^{\NN^2}$ is  reversible or not?
\end{question}

\bibliographystyle{plain}
\bibliography{conjbib}{}

\end{document}